\documentclass[oneside,a4paper,12pt,notitlepage]{article}
\usepackage[left=0.8in, right=0.8in, top=1.5in, bottom=1.5in]{geometry}

\usepackage[T1]{fontenc} % codifica dei font
\usepackage[utf8]{inputenc} % lettere accentate da tastiera
\usepackage[english]{babel}
\usepackage{lipsum} % genera testo fittizio
\usepackage{lmodern}
\usepackage{amssymb}
\usepackage{amsthm}
\usepackage{bm}
\usepackage{mathtools}
\usepackage{braket}
\usepackage{esint}
\usepackage{enumitem}
\usepackage{booktabs}
\usepackage{graphicx}
\usepackage{tikz}
\usetikzlibrary{patterns}
\usepackage{multicol}
\usepackage{caption}
\usepackage[skins,theorems]{tcolorbox}
\tcbset{highlight math style={enhanced,
		colframe=black,colback=white,arc=0pt,boxrule=1pt}}
\newtheorem{theorem}{Theorem}[section]
\newtheorem{corollary}{Corollary}[section]

\newtheorem{lemma}[theorem]{Lemma}%[section]
\newtheorem{example}{Example}

\theoremstyle{definition}
\newtheorem{definition}[theorem]{Definition}%[section]
\newtheorem{remark}{Remark}%[section]

\newcommand{\R}{\mathbb R}

\newcommand{\h}{\mathcal{H}^{n-1}}

\def\dfrac{\displaystyle\frac}
\def\dint{\displaystyle\int}

\DeclareMathOperator{\sgn}{sgn}
\makeatletter
\newcommand{\myfootnote}[2]{\begingroup
	\def\@makefnmark{}%
	\addtocounter{footnote}{-1}%
	\footnote{\textbf{#1} #2}
	\endgroup}
\makeatother

\usepackage{hyperref}
\hypersetup{linktoc=none, bookmarksnumbered, colorlinks=true, linkcolor=red}

\title{A Talenti comparison result for a class of Neumann \\  boundary value problems}
\author{A. Celentano, C. Nitsch, C. Trombetti}
\date{}

\newcommand{\Addresses}{{% additional braces for segregating \footnotesize 
  \bigskip 
  \footnotesize 
 
  \medskip 
 
  \noindent\textit{E-mail address}, A.~Celentano: \texttt{antonio.celentano2@unina.it}\\ 
    \noindent\textit{E-mail address}, C.~Nitsch: \texttt{c.nitsch@unina.it}\\
    \noindent\textit{E-mail address}, C.~Trombetti: \texttt{cristina@unina.it}
   \medskip 
     
    \noindent\textsc{Dipartimento di Matematica e Applicazioni ``R. Caccioppoli'', Universit\`a degli studi di Napoli Federico II, Via Cintia, Complesso Universitario Monte S. Angelo, 80126 Napoli, Italy.}

    \par\nopagebreak 

}}

\begin{document}
\maketitle

\begin{abstract} 
\noindent In this paper, we establish a comparison principle in terms of Lorentz norms and point-wise inequalities between a positive solution $u$ to the Poisson equation with non-homogeneous Neumann boundary conditions and a specific positive solution $v$ to the Schwartz symmetrized problem, which is related to $u$ through an additional boundary condition.\\

 

\end{abstract}
\Addresses 

\section{Introduction}
In the paper \cite{Tal}, Talenti proved a key result in the context of symmetrization techniques, which are tools used to deduce information about a problem, through the study of a simplified one. Given an open and bounded set $\Omega\subset\mathbb{R}^n$ and a function $f$ such that $f\in L^{\frac{2n}{n+2}}(\Omega)$ if $n>2$, $f\in L^p(\Omega)$ for some $p>1$ if $n=1$, a comparison principle is established between the solutions of the two following problems 
\begin{equation}\label{tprob}
\left\{
\begin{array}{ll}
-\Delta u= f  & \mbox{in $\Omega$},\\\\
u=0 & \mbox{on $\partial\Omega$},
\end{array}
\right.\qquad
\left\{
\begin{array}{ll}
-\Delta v= f^\sharp & \mbox{in $\Omega^\sharp$},\\\\
v=0 & \mbox{on $\partial\Omega^\sharp$},
\end{array}
\right.
\end{equation}
where $\Omega^\sharp$ denotes the ball, centered at the origin, with the same Lebesgue measure $\mathcal{L}^n$ as $\Omega$ and $f^\sharp$ is the decreasing Schwarz rearrangement of $f$ (see Section 2). Specifically, Talenti showed that $u^\sharp\leq v$ holds $\mathcal{L}^n$ almost everywhere in $\Omega^\sharp$.\\
%similar comparison results were established for more general operators with different boundary conditions.

\noindent Afterwards, many generalizations in different directions appeared in the literature. Under Dirichlet boundary conditions, nonlinear elliptic operators are considered in \cite{Tal2}, parabolic operators in \cite{ALT}, anisotropic elliptic operators in \cite{Alfetr} and higher order operators in \cite{T2, AB2}. Different boundary conditions were taken into account in \cite{ANT}, where the authors established a comparison result between the Lorentz norms of the solutions of the elliptic problems, obtained by replacing the Dirichlet conditions in \eqref{tprob} with a Robin one. In the planar case, if $f\equiv 1$ in $\Omega$, they are able to recover a point-wise inequality as the one proved by Talenti. The proof is based on the isoperimetric inequality, rearrangement properties and a careful analysis of superlevel sets of the solutions, which can touch the boundary in the case of Robin conditions. Further extensions are possible. In \cite{San}, the anisotropic case is treated; in \cite{AGM}, the authors deal with the $p$-Laplace operator; in \cite{CGNT}, with the Hermite operator and in \cite{ACNT}, with mixed Robin and Dirichlet boundary conditions.\\

\noindent In the end, we mention some results concerning the Neumann boundary conditions. In \cite{MSP}, the authors proved a point-wise comparison between the Schwarz rearrangements of positive and negative parts of a solution of an elliptic problem with homogeneous Neumann condition and the solutions of two symmetrized problems, both defined on a ball having half measure of $\Omega$, with Dirichlet boundary conditions, obtained by taking the Schwarz rearrangements of the positive and negative parts of $f$. Comparisons with solutions to problems that are not obtained by symmetrizing the original PDE are considered in \cite{Fe, FM}. In \cite{L2}, the author established a comparison between the $L^p$ norm of a solution of Poisson equations with homogeneous Neumann boundary conditions and a solution of the problem with cap-symmetrized data when $\Omega$ is a ball or a spherical shell. In \cite{Lang_rob}, similar results are proved also for pure Robin condition and mixed Robin and Neumann boundary conditions. Other kinds of rearrangements have been investigated, for instance, in \cite{L1, Brock2}. Eventually, we mention the work \cite{Ci}, where comparison principles for Neumann problems are established on Riemannian manifolds.\\

\noindent Using the methods presented in \cite{ANT}, we provide in this paper a comparison result in terms of Lorentz norms (see Definition \ref{ln}) and point-wise inequalities between a positive solution $u$ to the Poisson equation with non-homogeneous Neumann boundary conditions and a particular positive solution $v$ to the Schwartz symmetrized problem, which is linked to $u$ through an appropriate boundary condition.\\

\noindent For $n\geq 2$, let $\Omega\subset \mathbb{R}^n$ be an open, bounded and connected set with Lipschitz boundary. Given $f\in L^r(\Omega)$ with $r=\max\{2, 2^{-1}n\}$, we consider the following  problem
\begin{equation}\label{neumprob}
\left\{
\begin{array}{ll}
-\Delta u= f  & \mbox{in $\Omega$},\\\\
\dfrac{\partial u}{\partial \nu}=c & \mbox{on $\partial\Omega$,}
\end{array}
\right.
\end{equation}
where $\nu$ denotes the outer unit normal to $\partial\Omega$ and $c$ is a real constant.
A function $u \in W^{1,2}(\Omega)$ is a weak solution to \eqref{neumprob} if
\begin{equation}
\label{weaksol}
\int_{\Omega} \nabla u \cdot \nabla \phi \,dx -c\int_{\partial \Omega} \phi \,d\mathcal{H}^{n-1}= \int_{\Omega} f \phi \,dx, \quad \forall \phi \in H^1(\Omega).
\end{equation}
It is known \cite[Proposition 7.7]{Tay} that \eqref{neumprob} admits weak solutions, all defined modulo constants, if and only if the following compatibility condition is satisfied 
\begin{equation}\label{compcond}
c=-\frac{1}{\mathcal{H}^{n-1}(\partial\Omega)}\int_{\Omega} f\,dx.
\end{equation}
A Lipschitz domain satisfies the exterior cone property \cite[Theorem 1.2.2.2]{Gri}, then standard results \cite[Theorems 8.22, 8.27]{gilbtrud} imply that a solution $u$ to \eqref{neumprob} is bounded on $\Omega$ and locally Holder continuous. Hence, it exists $k_0\in \mathbb{R}$ such that for $k>k_0$, $u_k=u+k$ is a positive weak solution to \eqref{neumprob}. We take into account the Schwartz symmetrized problem
\begin{equation}\label{symm}
\left\{
\begin{array}{ll}
-\Delta v= f^\sharp & \mbox{in $\Omega^\sharp$},\\\\
\dfrac{\partial v}{\partial \nu}=c^\star & \mbox{on $\partial\Omega^\sharp$},
%\dint_{\partial \Omega^\sharp}v\,d\mathcal{H}^{n-1}=\gamma_1^\star,
\end{array}
\right.
\end{equation}
%which additionally satisfies an other boundary condition.
By applying rearrangement properties (see Section 2), we deduce from the compatibility condition in problems \eqref{neumprob} and \eqref{symm} that
\begin{equation}\label{cstar}
c^\star=-\frac{1}{\mathcal{H}^{n-1}(\partial\Omega^\sharp)}\int_{\Omega^\sharp} f^\sharp\,dx\leq -\frac{1}{\mathcal{H}^{n-1}(\partial\Omega^\sharp)}\int_{\Omega} f\,dx=\frac{\mathcal{H}^{n-1}(\partial\Omega)}{\mathcal{H}^{n-1}(\partial\Omega^\sharp)}c.
\end{equation}
From now on, we assume
\begin{align}\label{fposmed}
    \int_{\Omega} f\,dx>0,
\end{align}
which implies $c, c^*<0$. Moreover, the isoperimetric inequality ensures that $c^*\leq c$. A solution to problem \eqref{symm} is radial and non-increasing along the radius. A comparison principle can be established between a positive weak solution $u$ to \eqref{neumprob} and the unique positive weak solution $v\in W^{1,2}(\Omega^\sharp)$ to problem \eqref{symm} which satisfies one of the following equalities
\begin{equation}\label{cond_1}
  \frac{c^\star}{c}\int_{\partial \Omega}u\,d\mathcal{H}^{n-1}=\int_{\partial \Omega^\sharp}v\,d\mathcal{H}^{n-1},
\end{equation}
\begin{equation}\label{cond_2}
   \frac{c^\star}{c}\int_{\partial \Omega}u^2\,d\mathcal{H}^{n-1}=\int_{\partial \Omega^\sharp}v^2\,d\mathcal{H}^{n-1}.
\end{equation}
For sign changing solutions, some counterexamples are possible (see Example \ref{ex5} in Section 4). \\ 
We state our main results.
\begin{theorem}\label{th_main_f}
Let $u$ be a positive solution to problem \eqref{neumprob}. If $v$ is the positive solution to problem \eqref{symm} satisfying the equality \eqref{cond_1}, then
\begin{equation}\label{compL1}\|u\|_{L^{p,1}(\Omega)}\le \|v\|_{L^{p,1}(\Omega^\sharp)}\quad \mbox{for all } 0< p\le\frac{n}{2n-2};\end{equation}
If $v$ is the positive solution to problem \eqref{symm} satisfying the equality \eqref{cond_2}, then
\begin{equation}\label{compL1-2}
\|u\|_{L^{p,1}(\Omega)}\le \|v\|_{L^{p,1}(\Omega^\sharp)}\quad \mbox{for all } 0< p\le\frac{n}{2n-2},
\end{equation}
\begin{equation}\label{compL2}\|u\|_{L^{2p,2}(\Omega)}\le \|v\|_{L^{2p,2}(\Omega^\sharp)} \quad \mbox{for all } 0< p\le\frac{n}{3n-4}.\end{equation}
\end{theorem}
\begin{theorem}\label{th_main_1}
Let assume that $f\equiv 1$ in $\Omega$. Let $u$ be a positive solution to problem \eqref{neumprob}. If $v$ is the positive solution to problem \eqref{symm} satisfying the equality \eqref{cond_1}, it holds
$$u^\sharp(x)\le v(x) \quad \forall\, x \in \>\mbox{in $\Omega^\sharp$},\quad \>\mbox{for $n=2$}.$$
When $n\ge 3$, we have
$$\|u\|_{L^{p,1}(\Omega)}\le \|v\|_{L^{p,1}(\Omega^\sharp)}\quad \mbox{for all } \>0< p\le\frac{n}{n-2}.$$
If $v$ is the positive solution to problem \eqref{symm}  satisfying the equality \eqref{cond_2}, it holds
$$u^\sharp(x)\le v(x) \quad \forall\,x \in \>\mbox{in $\Omega^\sharp$}, \quad \>\mbox{for $n=2$}$$
When $n\ge 3$, we have
$$\|u\|_{L^{p,1}(\Omega)}\le \|v\|_{L^{p,1}(\Omega^\sharp)}\quad \mbox{and}\quad \|u\|_{L^{2p,2}(\Omega)}\le \|v\|_{L^{2p,2}(\Omega^\sharp)} \quad \mbox{for all } \>0< p\le\frac{n}{n-2}.$$
\end{theorem}
\noindent In the hypothesis of Theorem \ref{th_main_f}, when the dimension is $n=2$, we recover the $L^1$ comparison under both conditions \eqref{cond_1} and \eqref{cond_2}, but the $L^2$ comparison only under condition \eqref{cond_2}. A stronger result occurs in the case $f\equiv 1$ in $\Omega$. From Theorem \ref{th_main_1}, we get $\|u\|_{L^1(\Omega)}\leq \|v\|_{L^1(\Omega^\sharp)}$ and $\|u\|_{L^2(\Omega)}\leq \|v\|_{L^2(\Omega^\sharp)}$ in any dimension. The former holds true under both conditions \eqref{cond_1} and \eqref{cond_2}, whereas the latter only under condition \eqref{cond_2}.\\

\noindent This paper is organized as follows. In Section 2, we introduce definitions and some properties of rearrangements. In Section 3, we prove Theorems \ref{th_main_f} and \ref{th_main_1}. In Section 4, we discuss some generalizations and the optimality of our results, providing several examples.  
\section{Notations and Preliminaries}
We review the definitions of decreasing and Schwartz rearrangements, using \cite{leoni} as reference.
\begin{definition}
\label{rearrangement}
Let $h: x \in \Omega \rightarrow [0, +\infty)$ be a measurable function, then the decreasing rearrangement $h^*$ of $h$ is defined as 
\[
h^*(s) = \inf\{t  \in\R : |\{x\in\Omega: h(x)>t\}| < s\}, \quad s \in [0,|\Omega|].
\]
Let $\Omega^\sharp$ be the ball centered at the origin and having the same measure as $\Omega$. The Schwartz rearrangement of $h$ is defined as follows
\[
h^\sharp(x) = h^*(\omega_n |x|^n), \quad x \in \Omega^\sharp.
\]
where $\omega_n$ is the measure of the unit ball in $\mathbb{R}^n$.
\end{definition}
\begin{remark}\label{sr}
    Given a measurable function $h:\Omega\rightarrow \R$,  we define the decreasing rearrangement $h^\star$ of $h$ to be the decreasing rearrangement of $|h|$, that is $h^\star := (|h|)^\star$. In turn, we define the Schwartz rearrangement $h^\sharp$ of $h$ to be the Schwartz rearrangement of $|h|$.
\end{remark}
\noindent The functions $|h|$, $h^*$ and $h^\sharp$ are equi-distributed, that is 
$$
|\{x\in\Omega: |h(x)|>t\}| = |\{s\in (0,|\Omega|: h^*(s)>t\}|  = |\{x\in\Omega^\sharp: h^\sharp(x)>t\}|,\quad t\in\R.
$$
As a consequence, if $h\in L^p(\Omega)$, $ 1 \le p \le \infty$, then $h^*  \in L^p(0,|\Omega|)$,  $h^\sharp \in L^p(\Omega^\sharp)$ and $$||h||_{L^p(\Omega)}=||h^*||_{L^p(0,|\Omega|)}=||h^\sharp||_{L^p(\Omega^\sharp)} .$$
%\item if $u,v\in L^p(\Omega)$, $ 1 \le p \le \infty$, then 
%$$
%||u^\ast-v^\ast||_{L^p(0,|\Omega|)} \le ||u-v||_{L^p(\Omega)}.
%$$
%\item if $|h(x)| \le |g(x)|$ for a.e. $x \in \Omega$, then $h^*(s) \le g^*(s)$, $s \in (0, |\Omega|)$;
%
%\item for any $c \in \R$, $(h+c)^\ast(s)=u^*(s)+c$, $s \in (0,|\Omega|)$;
%\end{enumerate}
\noindent We shall make use of the following fundamental inequality involving rearrangements, which is the Hardy-Littlewood inequality. Given $h\in L^p(\Omega)$ and $g\in L^{q}(\Omega)$ such that $\frac{1}{p}+\frac{1}{q}=1$, then
\begin{equation}  \label{hl}
 \int_\Omega h(x)g(x)dx \le
\int_0^{|\Omega|}h^*(s)g^*(s)ds.
\end{equation}
%and equality holds if the following inclusions are true (see \cite{ALT}, \cite
%{CR})
%\begin{equation*}
%\{x \in \Omega: \>  f(x) >f^\ast(\mu_u(t))\} \subset \{x \in \Omega: \>
%|u(x)|>t\} \subset \{x \in \Omega: \>  f(x)  \ge f^\ast(\mu_u(t))\}
%\end{equation*}
%except for a set of measure zero. 
The Hardy-Littlewood inequality is useful in evaluating the integral of a function $h\in L^p(\Omega)$, $1 \le p \le +\infty$ on the level sets of a measurable function $u$, which is a common task when applying rearrangement theory to the study of partial differential equations. By choosing $g=\chi_{\{u>t\}}$ in \eqref{hl}, we get
\begin{equation}  \label{fu}
\int_{\{u>t\}} h(x) dx \le \int_0^{\mu(t)}h^\ast(s)ds.
\end{equation}
If we take $h=u \ge 0$ in \eqref{fu}, it follows
\begin{equation}\label{uu}
\int_{\{u>t\}} u(x) dx = \int_0^{\mu(t)}u^\ast(s)ds.
\end{equation}
For $U\subset \Omega$, we define 
\[
\partial_i U := \partial U \cap \Omega, \quad \partial_e U:= \partial U \cap \partial\Omega.
\]
Let $u$ and $v$ be solutions to problems \eqref{neumprob} and \eqref{symm} respectively. For $t\in\R$, we denote by
\[
U_t=\{x \in \Omega: |u(x)|>t\}, \quad \mu(t) = |U_t|.
\]
\[
V_t=\{x \in \Omega^\sharp: |v(x)|>t\}, \quad \phi(t) = |V_t|.
\]
\noindent Finally, we recall the definition of Lorentz space. 
\begin{definition}\label{ln}
    For $0<p<\infty$ and $0<q\le\infty$, the Lorentz space $L^{p,q}(\Omega)$ consists of all measurable functions $g$ in $\Omega$ such that it is finite the quantity 
\[
\|g\|_{L^{p,q}(\Omega)}=
\left\{
\begin{array}{ll}
\displaystyle p^\frac1q\left( \int_0^\infty t^q\left| \left\{ x\in\Omega:|g(x)|> t \right\}\right|^\frac{q}{p}\frac{dt}{t} \right)^\frac1q & 0<q<\infty,\\\\
\displaystyle \sup_{t>0}\left(t^p\left| \left\{ x\in\Omega:|g(x)|> t \right\} \right|\right) & q=\infty.
\end{array}
\right.
\]
\end{definition}
\noindent When $p=q$, the Lorentz space $L^{p,q}(\Omega)$ coincides with the $L^p(\Omega)$ space and $\|g\|_{L^{p,p}(\Omega)}=\|g\|_{L^{p}(\Omega)}$.
\section{Proof of the main Theorems}
Several lemmas are necessary to provide proofs of Theorems \eqref{th_main_f} and \eqref{th_main_1}.
\begin{lemma}[Gronwall]\label{lem_Gronwall}
Let $\xi(\tau):[\tau_0,+\infty)\rightarrow\mathbb{R}$ be a continuously differentiable function %for all $0<t_0\le t\le t_1$, 
satisfying, for some non negative constant $C$, the following differential inequality  
\[
\tau \xi'(\tau) \le \xi(\tau)+C \qquad \mbox{for all $\tau\ge\tau_0>0$.}
\] 
Then we have
\begin{itemize}
\item[(i)]\[ \xi(\tau) \le \tau\frac{\xi(\tau_0)+C}{\tau_0}-C \qquad \mbox{for all $\tau\ge\tau_0$},\] 
\item[(ii)]\label{ineq_Gronwall} \[ \xi'(\tau) \le \frac{\xi(\tau_0)+C}{\tau_0} \qquad \mbox{for all $\tau\ge\tau_0$}.\]
\end{itemize}
\end{lemma}
\noindent The following analysis of the superlevel sets of positive solutions to problems \eqref{neumprob} and \eqref{symm}, which may touch the boundary of $\Omega$ and $\Omega^\sharp$ respectively, is critical for the comparison results.
\begin{lemma}\label{lem_boundary_1}
Let $u$ and $v$ be positive solutions to problems \eqref{neumprob} and \eqref{symm} respectively. For a.e. $t>0$, we have
\begin{equation}\label{eq_fundamental}\gamma_n \phi(t)^\frac{2n-2}{n}=\left(-\phi'(t)-\frac{1}{c}\int_{\partial_e V_{t}}d\mathcal{H}^{n-1}\right)\int_0^{\phi(t)}f^*(s)ds
\end{equation}
and
\begin{equation}\label{ineq_fundamental}\gamma_n \mu(t)^\frac{2n-2}{n}\le\left(-\mu'(t)-\frac{1}{c}\int_{\partial_e U_{t}}d\mathcal{H}^{n-1}\right)\int_{0}^{\mu(t)}f^*(s)\,ds,
\end{equation}
where $\gamma_n=n^2\omega_n^{-\frac{2}{n}}$.
\end{lemma}
\begin{proof}
Given $t>0$ and $h>0$, we choose the test function in \eqref{weaksol} 
\begin{equation}\nonumber 
\varphi_h^i(x)= \left\{
\begin{array}{ll}
0 & \mbox{if $0<u<t$},\\\\
h & \mbox{if $u> t+h$},\\\\
u-t  &\mbox{if $t<u<t+h$}.
\end{array}
\right.
\end{equation}
Then,
\begin{equation}\nonumber 
\begin{array}{ll}
 \displaystyle\int_{U_t \setminus  U_{t+h}} |\nabla u|^2 \, dx -c h \displaystyle\int_{\partial_e U_{t+h}}\,&d\mathcal{H}^{n-1}-c\displaystyle\int_{\partial_e U_{t} \setminus \partial_e U_{t+h}} (u-t) \,d\mathcal{H}^{n-1}=  \\\\
& \displaystyle\int_{U_t \setminus U_{t+h}} f (u-t) \, dx + h \displaystyle\int_{U_{t+h} } f \, dx,
\end{array}
\end{equation}
dividing by $h$ and letting $h$ go to $0$, an application of the coarea formula shows that for a.e. $t>0$
\begin{equation}\nonumber 
\int_{\partial U_t} g(x) \,d\mathcal{H}^{n-1} = \int_{\partial_i U_t} |\nabla u| \, d\mathcal{H}^{n-1} -c \int_{\partial_e U_{t}} d\mathcal{H}^{n-1} = \int_{U_t} f\, dx,
\end{equation}
where
\[
g(x) = \left\{ \begin{array}{ll}
|\nabla u| & \mbox{if $x \in \partial_i U_t$,}\\\\
-c & \mbox{if $x \in \partial_e U_{t}$.}\\
\end{array}
\right.
\]
Using the isoperimetric inequality, we get that for a.e. $t>0$
\begin{align*}
    &\gamma_n \mu(t)^{\frac{2n-2}{n}}\leq \left(\mathcal{H}^{n-1}(U_t)\right)^2  = \left(\int_{\partial U_t}\,d\mathcal{H}^{n-1}\right)^2\\ 
    &\leq \left(\int_{\partial U_t}|g|\,d\mathcal{H}^{n-1}\right)\left(\int_{\partial U_t}\frac{1}{|g|}\,d\mathcal{H}^{n-1}\right)\\ 
    &\leq \int_{U_t}f\,dx\left(-\mu'(t)-\frac{1}{c}\int_{\partial_e U_{t}}d\mathcal{H}^{n-1}\right)\leq \int_{0}^{\mu(t)}f^*(s)\,ds\left(-\mu'(t)-\frac{1}{c}\int_{\partial_e U_{t}}d\mathcal{H}^{n-1}\right)
\end{align*}
For the function $v$, all the previous inequalities hold as equalities.
\end{proof}
\begin{remark}
From now on, we denote by $$v_m=\inf_{\Omega^\sharp} v, \>u_m=\inf_{\Omega} u.$$ Assuming hypothesis \eqref{fposmed}, whether any condition between \eqref{cond_1} and \eqref{cond_2} is in force, it holds
\begin{equation}\label{ineq_min}
u_m\le v_m,
\end{equation} 
Let us assume hypothesis \eqref{cond_1}, then 
\begin{equation}\nonumber
\frac{1}{c^*}v_m \mathcal{H}^{n-1}(\partial\Omega^\sharp)=\frac{1}{c^*}\int_{\partial\Omega^\sharp}v(x)\,d\mathcal{H}^{n-1}
= \frac{1}{c}\int_{\partial\Omega}u(x)\,d\mathcal{H}^{n-1}\leq \frac{1}{c}u_m \mathcal{H}^{n-1}(\partial\Omega)\leq 0.
\end{equation}
From the compatibility conditions in problems \eqref{neumprob} and \eqref{symm}, we get
\begin{equation}\nonumber
\frac{c^\star}{c}=\frac{\mathcal{H}^{n-1}(\partial\Omega)}{\mathcal{H}^{n-1}(\partial \Omega^\sharp)}\left(\int_{\Omega}f\,dx\right)^{-1}\int_{\Omega}|f|\,dx 
\end{equation}
and \eqref{ineq_min} follows from the isoperimetric inequality. Assuming hypothesis \eqref{cond_2}, a similar argument can be made. We also observe that  
\begin{equation}\label{ineq_iniziale}
\mu(t)\le \phi(t)=|\Omega|\qquad \mbox{for all $0\le t \le v_m$.}
\end{equation}
As an important consequence of the reasoning above, inequality \eqref{ineq_iniziale} holds strictly for some $0\le t \le v_m$, unless $\Omega$ is a ball and $f$ is a non negative function.
\end{remark}
\noindent We can estimate the boundary integral on the right-hand side of \eqref{eq_fundamental} and \eqref{ineq_fundamental}, using two fundamental lemmas.
\begin{lemma}\label{lem_boundary_2}
Let $u$ and $v$ be positive solutions to problems \eqref{neumprob} and \eqref{symm} respectively. For all $t\ge v_m$ we have
\begin{equation}\label{boudintL1rad}
\int_0^t \int_{\partial_e V_\tau}\, d\mathcal{H}^{n-1}\,d\tau = \int_{\partial \Omega^\sharp}v(x)\,d\mathcal{H}^{n-1},
\end{equation}
while
\begin{equation}\label{boudintL1}
\int_0^t \int_{\partial_e U_\tau}\, d\mathcal{H}^{n-1}\,d\tau \leq \int_{\partial \Omega}u(x)\,d\mathcal{H}^{n-1}.
\end{equation}
\end{lemma}
\begin{proof}
By Fubini's theorem, we have
\begin{align*}
\int_0^{+\infty} \int_{\partial_e U_\tau}\, d\mathcal{H}^{n-1}\,d\tau = \int_{\partial \Omega} \left(\int_0^{u(x)}\,d\tau\, \right)d\mathcal{H}^{n-1}
= \int_{\partial \Omega}u(x)\,d\mathcal{H}^{n-1}.
\end{align*}
Analogously,
\begin{align*}
\int_0^{+\infty} \int_{\partial_e V_\tau}\, d\mathcal{H}^{n-1}\,d\tau = \int_{\partial \Omega^\sharp} \left(\int_0^{v_m}\,d\tau\, \right)d\mathcal{H}^{n-1}
= \int_{\partial \Omega^\sharp}v(x)\,d\mathcal{H}^{n-1}.
\end{align*}
A trivial inequality for $t\ge 0$ is
\[
\int_0^{t} \int_{\partial_e U_\tau}\, d\mathcal{H}^{n-1}\,d\tau \leq\int_0^{+\infty} \int_{\partial_e U_\tau}\, d\mathcal{H}^{n-1}\,d\tau,
\]
while we observe that for $t\ge v_m=\min_{\partial\Omega^\sharp} v$ then $\partial V_t \cap \partial \Omega^\sharp = \emptyset$ and
\[
\int_0^{t} \int_{\partial_e V_\tau}\, d\mathcal{H}^{n-1}\,d\tau=\int_0^{+\infty} \int_{\partial_e V_\tau}\, d\mathcal{H}^{n-1}\,d\tau.
\]
\end{proof}
\begin{lemma}\label{lem_boundary_3}
Let $u$ and $v$ be positive solutions to problems \eqref{neumprob} and \eqref{symm} respectively. For all $t\ge v_m$ we have
\begin{equation}\label{boudintL2rad}
2\int_0^t \tau\left(\dint_{\partial_e V_\tau} \> d\h \right)\,d\tau =\int_{\partial \Omega}v^2\,d\mathcal{H}^{n-1},
\end{equation}
while
\begin{equation}\label{boudintL2}
2\int_0^t \tau\left(\dint_{\partial_e U_\tau} \>d\h \right)\,d\tau \le\int_{\partial \Omega}u^2\,d\mathcal{H}^{n-1}.
\end{equation}
\end{lemma}
\begin{proof}
By Fubini's theorem, we have
\begin{equation*}
\begin{array}{ll}
2\dint_0^\infty \tau\left(\dint_{\partial_e U_\tau}\>d\h \right)\,d\tau &= 2\dint_{\partial \Omega} \left(\int_0^{u(x)}\tau \>d\tau\, \right)d\h\\
&= \dint_{\partial \Omega} u^2(x) \>d\h.
\end{array}
\end{equation*}
Analogously,
\begin{equation*}
%\begin{array}{ll}
2\dint_0^\infty \tau\dint_{\partial_e V_\tau}\>d\h \,d\tau% &= \dint_{\partial \Omega} \int_0^{u(x)}\dfrac{\tau}{u(x)} \>d\tau\, d\h\\
%&= \dint_{\partial \Omega} \dfrac{u(x)}2 \>d\h
= \dint_{\partial \Omega} v^2(x) \>d\h.
%\end{array}
\end{equation*}
Reasoning as above, we obtain the thesis.
\end{proof}
\noindent For simplicity, we set
\begin{align*}
    \gamma_1&=\int_{\partial \Omega}u\,d\mathcal{H}^{n-1},\qquad \gamma_1^\star=\int_{\partial \Omega^\sharp}v\,d\mathcal{H}^{n-1},\\
    \gamma_2&=\int_{\partial \Omega}u^2\,d\mathcal{H}^{n-1},\qquad \gamma_2^\star=\int_{\partial \Omega^\sharp}v^2\,d\mathcal{H}^{n-1}.
\end{align*}
\begin{proof}[Proof of Theorem \ref{th_main_f}]
We start by proving inequality \eqref{compL1}, assuming condition \eqref{cond_1}. Let $0<p\le\frac{n}{2n-2}$. Multiplying \eqref{ineq_fundamental} by $\mu(t)^{\frac{1}{p}-\frac{2n-2}{n}}$, integrating from $0$ to some $\tau\ge v_m$ and applying Lemma \ref{lem_boundary_2}, we deduce that
\begin{align*}
\int_0^\tau\gamma_n \mu(t)^{\frac{1}{p}}\,dt\leq \int_0^\tau -\mu'(t)\mu(t)^\delta\left(\int_0^{\mu(t)}f^*(s)\,ds\right)\,dt-\frac{\gamma_1|\Omega|^\delta}{c}\int_0^{|\Omega|}f^*(s)\,ds,
\end{align*}
where we have set $\displaystyle\delta=\frac{1}{p}-\frac{2n-2}{n}$.
Using the monotonicity of $\mu$ and applying a change of variables \cite{leoni}[Theorem 6.14, Proposition 15.2], we get
\begin{align*}
\int_0^\tau\gamma_n \mu(t)^{\frac{1}{p}}\,dt\leq \int_{\mu(\tau)}^{|\Omega|} w^\delta\left(\int_0^w f^*(s)ds\right)\,dw-\frac{\gamma_1|\Omega|^\delta}{c}\int_0^{|\Omega|}f^*(s)\,ds.
\end{align*}
Taking $\tau\to+\infty$, it follows
\begin{align*}
\int_0^{+\infty}\gamma_n \mu(t)^{\frac{1}{p}}\,dt\leq \int_{0}^{|\Omega|} w^\delta\left(\int_0^w f^*(s)ds\right)\,dw-\frac{\gamma_1|\Omega|^\delta}{c}\int_0^{|\Omega|}f^*(s)\,ds.
\end{align*}
Arguing in the same way with the equality \eqref{eq_fundamental}, we have
\begin{align*}
\int_0^{+\infty}\gamma_n \phi(t)^{\frac{1}{p}}\,dt= \int_{0}^{|\Omega|} w^\delta\left(\int_0^w f^*(s)ds\right)\,dw-\frac{\gamma_1^\star|\Omega|^\delta}{c^\star}\int_0^{|\Omega|}f^*(s)\,ds.
\end{align*}
Hence
\begin{align*}
\int_0^{+\infty}\gamma_n \mu(t)^{\frac{1}{p}}\,dt\leq\int_0^{+\infty}\gamma_n \phi(t)^{\frac{1}{p}}\,dt,
\end{align*}
which means
$$\|u\|_{L^{p,1}(\Omega)}\le \|v\|_{L^{p,1}(\Omega^\sharp)}.$$
We will now assume condition \eqref{cond_2} and verify inequalities \eqref{compL1-2} and \eqref{compL2}.\\
Let $0 <p\le\frac{n}{2n-2}$. Multiplying \eqref{ineq_fundamental} by $t\mu(t)^{\frac{1}{p}-\frac{2n-2}{n}}$, integrating from $0$ to some $\tau\ge v_m$ and applying Lemma \ref{lem_boundary_3}, we get
\begin{equation*}
\begin{array}{rl}
\dint_0^\tau \gamma_n t\mu(t)^\frac1p \,dt \le& \dint_0^\tau-\mu'(t)t\mu(t)^{\delta}\left(\int_0^{\mu(t)}f^*(s)ds\right)\,dt-\dfrac{\gamma_2|\Omega|^\delta}{2c}\dint_0^{|\Omega|}f^*(s)ds,
\end{array}
\end{equation*}
where $\displaystyle\delta=\frac{1}{p}-\frac{2n-2}{n}$.
Since $\mu(t)$ is a monotone non increasing function, we have
\begin{equation}\label{ineq_basic}
\begin{array}{rl}
\dint_0^\tau \gamma_n t\mu(t)^\frac1p dt\le& \dint_0^\tau-t\mu(t)^{\delta}\left(\int_0^{\mu(t)}f^*(s)ds\right)\,d\mu(t)
-\dfrac{\gamma_2|\Omega|^\delta}{2c}\dint_0^{|\Omega|}f^*(s)ds.
\end{array}
\end{equation}
Therefore, it is possible to integrate by parts both sides of the last inequality. Introducing the function $F(\ell)=\dint_0^\ell w^\delta\left(\int_0^w f^*(s)ds\right)\,dw$, we obtain 
\begin{equation*}
\begin{array}{rl}
\tau F(\mu(\tau))+\tau\dint_0^\tau \gamma_n \mu(t)^\frac1p dt\le& \dint_0^\tau F(\mu(t))dt+\int_0^\tau\int_0^t \gamma_n\mu(t)^\frac{1}{p} dr\, dt\displaystyle-\dfrac{\gamma_2|\Omega|^\delta}{2c}\dint_0^{|\Omega|}f^*(s)ds.
\end{array}
\end{equation*}
Adopting the same notation of Lemma \ref{lem_Gronwall}, we set $$\xi(\tau)=\int_0^\tau F(\mu(t))dt+\int_0^\tau\left(\int_0^t \gamma_n\mu(r)^\frac{1}{p} dr\right)\, dt,$$ $C=-\dfrac{\gamma_2|\Omega|^\delta}{2c}\dint_0^{|\Omega|}f^*(s)ds$ and $\tau_0=v_m$. An application of Lemma \ref{lem_Gronwall} yields
\begin{equation}\label{ineq_afterGu}
\begin{array}{ll}
F(\mu(\tau))&+\dint_0^\tau \gamma_n \mu(t)^\frac1p dt\le\dfrac{1}{v_m}\Bigg(\int_0^{v_m} F(\mu(t))dt\\\\
&+\dint_0^{v_m}\int_0^t \gamma_n\mu(r)^\frac{1}{p} dr\, dt 
-\dfrac{\gamma_2|\Omega|^\delta}{2c}\dint_0^{|\Omega|}f^*(s)ds\,\Bigg).
\end{array}
\end{equation}
Arguing in the same way with the equality \eqref{eq_fundamental}, we have
\begin{equation}\label{ineq_afterGv}
\begin{array}{ll}
F(\phi(\tau))&+\dint_0^\tau \gamma_n \phi(t)^\frac1p dt= \dfrac{1}{v_m}\Bigg(\int_0^{v_m} F(\phi(t))dt\\\\
&+\dint_0^{v_m}\dint_0^t \gamma_n\phi(r)^\frac{1}{p} dr\, dt -\dfrac{\gamma^*_2|\Omega|^\delta}{2c^*}\dint_0^{|\Omega|}f^*(s)ds\,\Bigg)=\\\
& F(|\Omega|) + \dfrac{\gamma_n |\Omega|^\frac{1}{p} v_m}{2} -\dfrac{\gamma^*_2|\Omega|^\delta}{2 v_m c^*}\dint_0^{|\Omega|}f^*(s)ds.
\end{array}
\end{equation}
Recalling inequality \eqref{ineq_iniziale} and using the monotonicity property of function $F$, a comparison between the right-hand sides in \eqref{ineq_afterGu} and in \eqref{ineq_afterGv} can be established, yielding
$$F(\phi(\tau))+\int_0^\tau \gamma_n \phi(t)^\frac1p\ge F(\mu(\tau))+\int_0^\tau \gamma_n \mu(t)^\frac1p.$$
Passing to the limit as $\tau\to \infty$, we get
$$\int_0^\infty \mu(t)^\frac1p dt\le\dint_0^\infty \phi(t)^\frac1p dt,$$
and hence 
$$\|u\|_{L^{p,1}(\Omega)}\le \|v\|_{L^{p,1}(\Omega^\sharp)}.$$
Finally, we prove inequality \eqref{compL2}. It can be rewritten as follow 
$$ \int_0^\infty t\mu(t)^\frac1p dt\le \int_0^\infty t\phi(t)^\frac1p dt,\quad \forall\,p\in\left(0, \frac{n}{3n-4}\right].$$
In \eqref{ineq_basic}, we take the limit as $\tau\to\infty$, and then we integrate by parts the first term on the right-hand side to get
$$\int_0^\infty \gamma_n t\mu(t)^\frac1p dt \le \int_0^\infty F(\mu(t))dt-\dfrac{\gamma_2|\Omega|^\delta}{2c}\dint_0^{|\Omega|}f^*(s)ds.$$
On the other hand $$\int_0^\infty \gamma_n t\phi(t)^\frac1p dt=\int_0^\infty F(\phi(t))dt-\dfrac{\gamma^*_2|\Omega|^\delta}{2c^*}\dint_0^{|\Omega|}f^*(s)ds.$$
Hence, it is enough to show that 
\begin{equation}\label{ineq_desired}
\int_0^\infty F(\mu(t))dt\le \int_0^\infty F(\phi(t))dt.
\end{equation}
In order to prove inequality \eqref{ineq_desired}, we multiply \eqref{ineq_fundamental} by $t F(\mu(t)) \mu(t)^{-\frac{2n-2}{n}}.$
Since the function $F(\ell) \ell^{-\frac{2n-2}{n}}$ is non decreasing in $\ell$ $\mbox{for } 0< p\le\frac{n}{3n-4}$, an integration from $0$ to any $\tau\ge v_m$ yields
\begin{equation}\label{ineq_F}
\begin{array}{ll}
\dint_0^\tau \gamma_n tF(\mu(t))dt \le& \dint_0^\tau-t\mu^{-\frac{2n-2}{n}}F(\mu(t))\left(\dint_0^{\mu(t)}f^*(s)ds \right)\,d\mu(t)\\
&-F(|\Omega|)\dfrac{\gamma_2|\Omega|^{-\frac{2n-2}{n}}}{2c}\dint_0^{|\Omega|}f^*(s)ds.
\end{array}
\end{equation}
We now integrate by parts both sides of \eqref{ineq_F}. After setting $C=-F(|\Omega|)\dfrac{\gamma_2|\Omega|^{-\frac{2n-2}{n}}}{2c}\dint_0^{|\Omega|}f^*(s)ds$ and $H(\ell)=\dint_0^\ell w^{-\frac{2n-2}{n}}F(w)\int_0^{w}f^*(s)ds\,dw,$ we get
$$\tau\int_0^\tau\gamma_n F(\mu(t))dt+\tau H(\mu(\tau))\le \int_0^\tau\int_0^r\gamma_n F(\mu(z))dz\,dr+\int_0^\tau H(\mu(t))dt+C.$$
Lemma \ref{lem_Gronwall} can be applied with $\xi(\tau)=\dint_0^\tau\int_0^r\gamma_n F(\mu(z))dz\,dr+\int_0^\tau H(\mu(t))dt$ and $\tau_0=v_m$, obtaining
\begin{equation*}
\begin{array}{ll}
&\dint_0^\tau\gamma_n F(\mu(t))dt+H(\mu(\tau))\\\\
\le&\dfrac{1}{v_m}\left(\dint_0^{v_m}\dint_0^r\gamma_n F(\mu(z))dz\,dr+\int_0^{v_m} H(\mu(t))dt+C\right).
\end{array}
\end{equation*}
The above inequality holds as an equality whenever $\mu$ is replaced by $\phi$. Since \eqref{ineq_iniziale} is in force, it holds 
$$\int_0^\tau\gamma_n F(\mu(t))dt+ H(\mu(\tau))\le\int_0^\tau\gamma_n F(\phi(t))dt+ H(\phi(\tau))$$ and for $\tau \to \infty$ we get inequality \eqref{ineq_desired} which concludes the proof. 
\end{proof}
\begin{proof}[Proof of Theorem \ref{th_main_1}]
First, we prove the point-wise comparison in dimension $n=2$, assuming condition \eqref{cond_1}. We integrate both sides of inequality \eqref{ineq_fundamental} from $0$ to $\tau\ge v_m$. Observing that now 
 \begin{equation}\label{hl2}
\int_0^{\mu(t)}f^*(s)ds=\mu(t),
\end{equation}
we get
 \begin{equation*}
4 \pi \tau  \leq \int_0^\tau -\mu'(t) \,dt-\dfrac{\gamma_1}{c}\quad \mbox{for } \tau\ge v_m.
\end{equation*}
The equality holds whenever $\mu$ is replaced by $\phi$
 \begin{equation*}
4 \pi \tau  =\int_0^\tau -\phi'(t) \,dt-\dfrac{\gamma_1^\star}{c^\star}\quad \mbox{for } \tau\ge v_m.
\end{equation*}
Therefore,
\begin{equation}\nonumber
\int_0^\tau  \,(-d \mu(t) )\geq \int_0^\tau  \,(-d \phi(t) ) \quad \mbox{for } \tau \ge v_m,
\end{equation}
which means 
\begin{equation}
\label{estimate7}
\mu(\tau) \leq \phi(\tau) \quad  \mbox{for } \tau\ge v_m.
\end{equation}
We get the thesis recalling inequality \eqref{ineq_iniziale}. Now we consider $n\ge3$. Inequality \eqref{ineq_fundamental} becomes
\begin{equation}\nonumber
\gamma_n \mu(t)^\frac{n-2}{n}\le\left(-\mu'(t)-\dfrac{1}{c} \int_{\partial_e U_t} \,d\mathcal{H}^{n-1} \right).
\end{equation}
Let $q\le\frac{n}{n-2}$. Multiplying \eqref{ineq_fundamental} by $\mu(t)^{\frac{1}{q}-\frac{n-2}{n}}$, integrating from $0$ to some $\tau\ge v_m$ and applying Lemma \ref{lem_boundary_2}, we deduce that
\begin{equation*}
\int_0^\tau \gamma_n \mu(t)^\frac1q \,dt \le \int_0^\tau-\mu'(t)\mu(t)^{\eta}dt-\dfrac{\gamma_1|\Omega|^{\eta}}{c},
\end{equation*}
where $\displaystyle\eta=\frac{1}{q}-\frac{n-2}{n}$.
Using the monotonicity of $\mu$ and applying a change of variables, we obtain
\begin{equation*}
\int_0^\tau \gamma_n \mu(t)^\frac1q dt\le \int_{\mu(\tau)}^{|\Omega|}w^n\,dw-\dfrac{\gamma_1|\Omega|^{\eta}}{c}.
\end{equation*}
Taking $\tau\to+\infty$
\begin{equation*}
\int_0^{+\infty} \gamma_n \mu(t)^\frac1q\, dt\le \int_{0}^{|\Omega|}w^n\,dw-\dfrac{\gamma_1|\Omega|^{\eta}}{c}.
\end{equation*}
Arguing in the same way with \eqref{eq_fundamental}, we get
\begin{equation*}
\int_0^{+\infty} \gamma_n \phi(t)^\frac1q\, dt= \int_{0}^{|\Omega|}w^n\,dw-\dfrac{\gamma_1^\star|\Omega|^{\eta}}{c^\star}
\end{equation*}
and the thesis follows.
We will now assume condition \eqref{cond_2} and prove the point-wise comparison in dimension $n=2$. We multiply by $t$ the inequality \eqref{ineq_fundamental} and we integrate from $0$ to $\tau\ge v_m$. Since inequality \eqref{hl2} holds, we have
\begin{equation*}
2 \pi \tau^2  \leq \dint_0^\tau -\mu'(t) t \,dt-\dfrac{\gamma_2}{2c}\quad \mbox{for } \tau\ge v_m.
\end{equation*}
At the same time, equality holds true whenever $\mu$ is replaced by $\phi$
 \begin{equation*}
2 \pi \tau^2  =\dint_0^\tau -\phi'(t) t \,dt-\dfrac{\gamma_2^*}{2c^*}\quad \mbox{for } \tau\ge v_m.
\end{equation*}
Then,
\begin{equation}\label{estimate6}
\dint_0^\tau  t \,(-d \mu(t) )\geq \dint_0^\tau  t \,(-d \phi(t) ), \quad  \mbox{for } \tau \ge v_m.
\end{equation}
An integration by parts gives
\begin{equation}
\label{estimate10}
\mu(\tau) \leq \phi(\tau), \quad  \mbox{for } \tau \ge v_m.
\end{equation}
Since \eqref{ineq_min} is in force, inequality \eqref{estimate10} follows for $ t\ge 0$ and the claim is proved. We consider $n\ge3$. Inequality \eqref{ineq_fundamental} reads as follows
\begin{equation}\nonumber
\gamma_n \mu(t)^\frac{n-2}{n}\le\left(-\mu'(t)-\dfrac{1}{c} \dint_{\partial_e U_t} \>d\h \right).
\end{equation}
Let $q\le\frac{n}{n-2}$. Multiplying \eqref{ineq_fundamental} by $t\mu(t)^{\frac{1}{q}-\frac{n-2}{n}}$, integrating from $0$ to some $\tau\ge v_m$ and applying Lemma \ref{lem_boundary_3}, we deduce that
\begin{equation*}
\int_0^\tau \gamma_n t\mu(t)^\frac1q \,dt \le \int_0^\tau-\mu'(t)t\mu(t)^{\eta}dt-\dfrac{\gamma_2|\Omega|^{\eta}}{2c}.
\end{equation*}
Here we have set $\displaystyle\eta=\frac{1}{q}-\frac{n-2}{n}$.
Since $\mu(t)$ is a monotone non increasing function, we can write
\begin{equation}\label{ineq_basic1}
\int_0^\tau \gamma_n t\mu(t)^\frac1q dt\le \int_0^\tau-t\mu(t)^{\eta}\,d\mu(t)-\dfrac{\gamma_2|\Omega|^{\eta}}{2c}.
\end{equation}
Integrating by parts the last inequality and introducing the function $G(\ell)=\dint_0^\ell w^\eta= \frac{\ell^{\eta+1}}{\eta+1} $, we obtain
\begin{equation*}
\tau G(\mu(\tau))+\tau\int_0^\tau \gamma_n \mu(t)^\frac1q dt\le \int_0^\tau G(\mu(t))dt +\int_0^\tau\int_0^t \gamma_n\mu(t)^\frac{1}{q} dr\, dt-\dfrac{\gamma_2|\Omega|^{\eta}}{2c}.
\end{equation*}
After setting $$\xi(\tau)=\int_0^\tau G(\mu(t))dt+\int_0^\tau\int_0^t \gamma_n\mu(t)^\frac{1}{q} dr\, dt,$$ $C=-\dfrac{\gamma_2|\Omega|^{\eta}}{2c}$ and $\tau_0=v_m$, Lemma \ref{lem_Gronwall} can be applied, yielding
\begin{equation}\label{ineq_1Gu}
\begin{array}{ll}
\displaystyle G(\mu(\tau))&+\displaystyle\int_0^\tau \gamma_n \mu(t)^\frac1q dt\le \frac{1}{v_m}\Bigg(\int_0^{v_m} G(\mu(t))dt\\\\
& \displaystyle+\int_0^{v_m}\int_0^t \gamma_n\mu(r)^\frac{1}{q} dr\, dt-\dfrac{\gamma_2|\Omega|^{\eta}}{2c}\Bigg).
\end{array}
\end{equation}
Arguing in the same way with equality \eqref{eq_fundamental}, we can conclude that
\begin{equation}\label{ineq_1Gv}
\begin{array}{ll}
\displaystyle G(\phi(\tau))&+\displaystyle\int_0^\tau \gamma_n \phi(t)^\frac1q dt= \frac{1}{v_m}\Bigg(\int_0^{v_m} G(\phi(t))dt\\\\
&\displaystyle+\int_0^{v_m}\int_0^t \gamma_n\phi(r)^\frac{1}{q} dr\, dt -\dfrac{\gamma_2^*|\Omega|^{\eta}}{2c^*}\Bigg).
\end{array}
\end{equation}
Recalling inequality \eqref{ineq_iniziale} and using the monotonicity property of function $G$, a comparison between the right-hand sides in \eqref{ineq_1Gu} and in \eqref{ineq_1Gv} can be established, yielding
$$G(\mu(\tau))+\int_0^\tau \gamma_n \mu(t)^\frac1q\le G(\phi(\tau))+\int_0^\tau \gamma_n \phi(t)^\frac1q.$$
Passing to the limit as $t\to \infty$ we get
$$\int_0^\infty \mu(t)^\frac1q dt\le\int_0^\infty \phi(t)^\frac1q dt$$
and hence 
$$\|u\|_{L^{q,1}(\Omega)}\le \|v\|_{L^{q,1}(\Omega^\sharp)}.$$
In order to conclude the proof, we have to show that 
$$ \int_0^\infty t\mu(t)^\frac1q dt\le \int_0^\infty t\phi(t)^\frac1q dt.$$
We consider the limit as $\tau\to\infty$ in \eqref{ineq_basic1}, and then we integrate by parts the first term on the right-hand side to obtain
$$\int_0^\infty \gamma_n t\mu(t)^\frac1q \le \int_0^\infty G(\mu(t))dt-\dfrac{\gamma_2|\Omega|^{\eta}}{2c}.$$
At the same time $$\int_0^\infty \gamma_n t\phi(t)^\frac1q =\int_0^\infty G(\phi(t))dt-\dfrac{\gamma_2^*|\Omega|^{\eta}}{2c^*}.$$
Therefore, it is enough to prove that 
\begin{equation}\label{ineq_desired1}
\int_0^\infty G(\mu(t))dt\le \int_0^\infty G(\phi(t))dt.
\end{equation}
To show inequality \eqref{ineq_desired1}, we multiply the inequality \eqref{ineq_fundamental} by $t G(\mu(t)) \mu(t)^{-\frac{n-2}{n}}.$
Since the function $G(\ell) \ell^{-\frac{n-2}{n}} = \ell^{\eta+ 2/n}$ is non decreasing in $\ell$ for $0<q\le\frac{n}{n-2}$, an integration from $0$ to any $\tau\ge v_m$ yields
\begin{equation}\label{ineq_G}
\begin{array}{ll}
\displaystyle\int_0^\tau \gamma_n tG(\mu(t))dt \le&\displaystyle \int_0^\tau-t\mu^{-\frac{n-2}{n}}G(\mu(t))\,d\mu(t)\\\\
&\displaystyle-G(|\Omega|)\dfrac{\gamma_2|\Omega|^{-\frac{2-n}{n}}}{2c}.
\end{array}
\end{equation}
We now integrate by parts both sides of \eqref{ineq_G}. After setting $C=-G(|\Omega|)\dfrac{\gamma_2|\Omega|^{-\frac{2-n}{n}}}{2c}$ and $J(\ell)=\dint_0^\ell w^{-\frac{n-2}{n}}G(w)\,dw$, we have
$$\tau\int_0^\tau\gamma_n G(\mu(t))dt+\tau J(\mu(\tau))\le \int_0^\tau\int_0^r\gamma_n G(\mu(z))dz\,dr+\int_0^\tau J(\mu(t))dt+C.$$
As above, we can apply Lemma \ref{lem_Gronwall} with $$\displaystyle \xi(\tau)=\int_0^\tau\int_0^r\gamma_n G(\mu(z))dz\,dr+\int_0^\tau J(\mu(t))dtdt,$$ and $\tau_0=v_m$, deducing that
$$\int_0^\tau\gamma_n G(\mu(t))dt+J(\mu(\tau))\le \frac{1}{v_m}\left(\int_0^{v_m}\int_0^r\gamma_n G(\mu(z))dz\,dr+\int_0^{v_m} J(\mu(t))dt+C\right).$$
The previous inequality holds as an equality whenever $\mu$ is replaced by $\phi$. Recalling inequality \eqref{ineq_iniziale}, it holds
$$\int_0^\tau\gamma_n G(\mu(t))dt+ J(\mu(\tau))\le\int_0^\tau\gamma_n G(\phi(t))dt+ J(\phi(\tau))$$ and for $\tau \to \infty$, we get inequality \eqref{ineq_desired1} which concludes the proof.
\end{proof}
\section{Further remarks}
The proofs of Theorems \ref{th_main_f}, \ref{th_main_1} can be generalized to the case of an open and bounded Lipschitz set $\Omega$ which has a finite number $k\geq 1$ of connected components $\Omega_j$. Since $\Omega$ is a Lipschitz set, the boundaries of its connected components must be disjoint. We generalize the boundary condition in problem \eqref{neumprob} with a piecewise constant one, by imposing $\mathcal{H}^{n-1}$-a.e on $\partial\Omega$, the equality between the normal derivative of a solution $u$ and $g(x)=\sum_{j=1}^k c_j\chi_{\partial\Omega_j}(x)$. We assume that the integral of the function $f$ on $\Omega_j$ is positive for every $j$. If the compatibility condition \eqref{compcond} is satisfied for every connected component $\Omega_j$, then the results of Theorems \ref{th_main_1}, \ref{th_main_f} can be recovered. We establish a comparison between a positive solution $u$ to the following problem
\begin{equation}\label{neumprob2}
\left\{
\begin{array}{ll}
-\Delta u= f  & \mbox{in $\Omega$},\\\\
\dfrac{\partial u}{\partial \nu}=g & \mbox{on $\partial\Omega$}.
\end{array}
\right.
\end{equation}
and the positive solution $v$ to the Schwartz symmetrized problem, having the same form of problem \eqref{symm}, which satisfies one of the following equalities
\begin{equation}\label{cond_1_g}
  \sum_{j=1}^k\frac{c^*}{c_j}\int_{\partial \Omega_j}u\,d\mathcal{H}^{n-1}=\int_{\partial \Omega^\sharp}v\,d\mathcal{H}^{n-1},
\end{equation}
\begin{equation}\label{cond_2_g}
   \sum_{j=1}^k\frac{c^*}{c_j}\int_{\partial \Omega_j}u^2\,d\mathcal{H}^{n-1}=\int_{\partial \Omega^\sharp}v^2\,d\mathcal{H}^{n-1}.
\end{equation}
\begin{corollary}
Let $u$ be a positive solution to problem \eqref{neumprob2}. If $v$ is the positive solution to problem \eqref{symm} satisfying the equality \eqref{cond_1_g}, then 
\begin{equation}\label{compL1_g}\|u\|_{L^{p,1}(\Omega)}\le \|v\|_{L^{p,1}(\Omega^\sharp)}\quad \mbox{for all } 0< p\le\frac{n}{2n-2};\end{equation}
If $v$ is the positive solution to problem \eqref{symm} satisfying the equality \eqref{cond_2_g}, then 
\begin{equation}\label{compL1-2_g}
\|u\|_{L^{p,1}(\Omega)}\le \|v\|_{L^{p,1}(\Omega^\sharp)}\quad \mbox{for all } 0< p\le\frac{n}{2n-2},
\end{equation}
\begin{equation}\label{compL2_g}\|u\|_{L^{2p,2}(\Omega)}\le \|v\|_{L^{2p,2}(\Omega^\sharp)} \quad \mbox{for all } 0< p\le\frac{n}{3n-4}.\end{equation}
\end{corollary}
\begin{corollary}
    Let assume that $f\equiv 1$ in $\Omega$. Let $u$ be a positive solution to problem \eqref{neumprob2}. If $v$ is the positive solution to problem \eqref{symm} satisfying the equality \eqref{cond_1_g}, it holds
$$u^\sharp(x)\le v(x) \quad \forall\, x \in \>\mbox{in $\Omega^\sharp$},\quad \>\mbox{for $n=2$}.$$
When $n\ge 3$, we have
$$\|u\|_{L^{p,1}(\Omega)}\le \|v\|_{L^{p,1}(\Omega^\sharp)}\quad \mbox{for all } \>0< p\le\frac{n}{n-2}.$$
If $v$ is the positive solution to problem \eqref{symm} satisfying the equality \eqref{cond_2_g}, it holds
$$u^\sharp(x)\le v(x) \quad \forall\,x \in \>\mbox{in $\Omega^\sharp$}, \quad \>\mbox{for $n=2$}$$
When $n\ge 3$, we have
$$\|u\|_{L^{p,1}(\Omega)}\le \|v\|_{L^{p,1}(\Omega^\sharp)}\quad \mbox{and}\quad \|u\|_{L^{2p,2}(\Omega)}\le \|v\|_{L^{2p,2}(\Omega^\sharp)} \quad \mbox{for all } \>0< p\le\frac{n}{n-2}.$$
\end{corollary}
\noindent In the end, we provide five examples. The first shows that under condition \eqref{cond_1_g}, we can't recover the comparison \eqref{compL2_g} in $L^2$ norm, even in dimension 2. 
\begin{example} Let $\Omega\subset\mathbb{R}^2$ be the union of two disjoint balls $B_1$ and $B_2$ of radius 1. Let $u, v\geq 0$ be the unique solution to the problems 
\begin{equation}\label{ex1}
\left\{
\begin{array}{ll}
-\Delta u=f  & \mbox{in $\Omega$},\\\\
\dfrac{\partial u}{\partial \nu}=-\frac{1}{2}\chi_{\partial B_1}-\frac{\epsilon}{2}\chi_{\partial B_2} & \mbox{on $\partial\Omega$},\\\\
u=\chi_{\partial B_1} & \mbox{on $\partial\Omega$},
\end{array}
\right.\qquad
\left\{
\begin{array}{ll}
-\Delta v= f^\sharp & \mbox{in $\Omega^\sharp$},\\\\
\dfrac{\partial v}{\partial \nu}=-\frac{\sqrt{2}}{4}(1+\epsilon) & \mbox{on $\partial\Omega^\sharp$},\\\\
v=\frac{1+\epsilon}{2} & \mbox{on $\partial\Omega^\sharp$},
\end{array}
\right.
\end{equation}
where $\epsilon>0, f= \chi_{B_1}+\epsilon\chi_{B_2}$. Then $$\|u\|_{L^2(\Omega)}^2-\|v\|_{L^2(\Omega^\sharp)}^2=\frac{\pi}{16}(4 +\ln2)-\frac{1}{32}\pi \epsilon(47 + \log(16))+o(\epsilon)$$
\end{example}
\begin{proof}
    The solutions $u$ and $v$ are positive for the maximum principle. A straightforward calculation shows that condition \eqref{cond_1_g} is satisfied but 
    \begin{equation*}
        \sum_{j=1}^2\frac{c^*}{c_j}\int_{\partial \Omega_j}u^2\,d\mathcal{H}^{n-1}>\int_{\partial \Omega^\sharp}v^2\,d\mathcal{H}^{n-1}.
    \end{equation*}
Moreover
    \begin{equation*}
    u_{| B_1}=\frac{5-r^2}{4},\qquad u_{| B_2}=\frac{\epsilon}{4}(1-r^2).
    \end{equation*}
    Since $\Omega^\sharp$ is a ball of radius $\sqrt{2}$, $f^\sharp(r)=1$ for $0\leq r< 1$ and $f^\sharp(r)=\epsilon$ for $1<r< \sqrt{2}$, we get
    \begin{equation*}
    v=\left\{
    \begin{array}{ll}
c_1-\frac{r^2}{4} & \mbox{for $0\leq r\leq 1$},\\\\
 c_2+c_3\ln{r}-\frac{\epsilon}{4}r^2& \mbox{for $1<r\leq \sqrt{2}$}.
\end{array}\right.
    \end{equation*}
    The solution $v$ is $C^1$ and by imposing the continuity of $v'$, $v$ and the second boundary condition in \eqref{ex1}, we obtain
        \begin{equation*}
    c_3=\frac{1}{2}(\epsilon-1),\quad c_2=\frac{1}{2}+\epsilon+\frac{1}{2}(1-\epsilon)\ln{\sqrt{2}},\quad c_1=\frac{3}{4}(1+\epsilon)+\frac{1}{2}(1-\epsilon)\ln{\sqrt{2}}.
    \end{equation*}
    By means of a first order expansion, we get
    \begin{align*}
       \|u\|_{L^2(\Omega)}^2=\frac{61}{48}\pi+o(\epsilon),\quad \|v\|_{L^2(\Omega^\sharp)}^2=\frac{\pi}{48}(49 - 3\log(2))+\frac{1}{32}\pi \epsilon(47 + \log(16))+o(\epsilon).
    \end{align*}
    In conclusion
        \begin{align*}
       \|u\|_{L^2(\Omega)}^2-\|v\|_{L^2(\Omega^\sharp)}^2=\frac{\pi}{16}(4 +\ln2)-\frac{1}{32}\pi \epsilon(47 + \log(16))+o(\epsilon).
    \end{align*}
\end{proof}
Next, we set the previous example in dimension 3. It shows that the $L^1$ comparison \eqref{compL1_g} becomes false in dimension greater than 2, if the function $f$ is not identically 1 in $\Omega$. 
\begin{example} Let $\Omega\subset\mathbb{R}^3$ be the union of two disjoint balls $B_1$ and $B_2$ of radius 1. Let $u, v\geq 0$ be the unique solutions to the problems 
\begin{equation}\label{ex2}
\left\{
\begin{array}{ll}
-\Delta u=f & \mbox{in $\Omega$},\\\\
\dfrac{\partial u}{\partial \nu}=-\frac{1}{3}\chi_{\partial B_1}-\frac{\epsilon}{3}\chi_{\partial B_2} & \mbox{on $\partial\Omega$},\\\\
u=\chi_{\partial B_1} & \mbox{on $\partial\Omega$},
\end{array}
\right.\qquad
\left\{
\begin{array}{ll}
-\Delta v= f^\sharp & \mbox{in $\Omega^\sharp$},\\\\
\dfrac{\partial v}{\partial \nu}=-\frac{1+\epsilon}{3\sqrt[3]{4}} & \mbox{on $\partial\Omega^\sharp$},\\\\
v=\frac{1+\epsilon}{2\sqrt[3]{2}} & \mbox{on $\partial\Omega^\sharp$},
\end{array}
\right.
\end{equation}
where $\epsilon>0, f= \chi_{B_1}+\epsilon\chi_{B_2}$. Then $$\|u\|_{L^1(\Omega)}-\|v\|_{L^1(\Omega^\sharp)}=\frac{14\pi}{9} - \frac{8\sqrt[3]{4}}{9}\pi - \frac{2\pi}{45}\epsilon - \frac{28\sqrt[3]{4}}{45}\pi\epsilon.$$
\end{example}
\begin{proof}
 The solutions $u$ and $v$ are positive for the maximum principle. A straightforward calculation shows that condition \eqref{cond_1_g} is satisfied. Moreover
    \begin{equation*}
    u_{| B_1}=\frac{7-r^2}{6},\qquad u_{| B_2}=\frac{\epsilon}{6}(1-r^2).
    \end{equation*}
    Since $\Omega^\sharp$ is a ball of radius $\sqrt[3]{2}$, $f^\sharp(r)=1$ for $0\leq r< 1$ and $f^\sharp(r)=\epsilon$ for $1<r< \sqrt[3]{2}$, we get
    \begin{equation*}
    v=\left\{
    \begin{array}{ll}
c_1-\frac{r^2}{6} & \mbox{for $0\leq r\leq 1$},\\\\
 c_2+\frac{c_3}{r}-\frac{\epsilon}{6}r^2& \mbox{for $1<r\leq\sqrt[3]{2}$}.
\end{array}\right.
    \end{equation*}
    The solution $v$ is $C^1$ and by imposing the continuity of $v'$, $v$ and the second boundary condition in \eqref{ex1}, we obtain
        \begin{equation*}
    c_3=\frac{1-\epsilon}{3},\quad c_2=\frac{(1+7\epsilon)}{6\sqrt[3]{2}},\quad c_1=\frac{1}{2} + \frac{1}{6\sqrt[3]{2}} - \frac{\epsilon}{2} +\frac{7\epsilon}{6\sqrt[3]{2}}.
    \end{equation*}
   Evaluating the $L^1$-norms, we get
    \begin{align*}
       \|u\|_{L^1(\Omega)}=\frac{64}{45}\pi+\frac{4\pi}{45}\epsilon,\quad \|v\|_{L^1(\Omega^\sharp)}=-\frac{4}{30}\pi+\frac{8\sqrt[3]{4}}{9}\pi+ 4\pi\left(\frac{1}{30}+ \frac{7\sqrt[3]{4}}{45}\right)\epsilon.
    \end{align*}
    In conclusion
        \begin{align*}
       \|u\|_{L^1(\Omega)}-\|v\|_{L^1(\Omega^\sharp)}=\frac{14\pi}{9} - \frac{8\sqrt[3]{4}}{9}\pi - \frac{2\pi}{45}\epsilon - \frac{28\sqrt[3]{4}}{45}\pi\epsilon.
    \end{align*}
\end{proof}
\noindent With the following example, we show that even under condition \eqref{cond_2_g} and dimension 2, an $L^p$ comparison for $p>2$ is not satisfied if $p$ is large enough.
\begin{example} Let $\Omega\subset\mathbb{R}^2$ be the union of two disjoint balls $B_1$ and $B_2$ of radius 1. Let $u, v\geq 0$ be the unique solutions to the problems 
\begin{equation}\label{ex3}
\left\{
\begin{array}{ll}
-\Delta u=f  & \mbox{in $\Omega$},\\\\
\dfrac{\partial u}{\partial \nu}=-\frac{1}{2}\chi_{\partial B_1}-\frac{\epsilon}{2}\chi_{\partial B_2} & \mbox{on $\partial\Omega$},\\\\
u=\chi_{\partial B_1} & \mbox{on $\partial\Omega$},
\end{array}
\right.\qquad
\left\{
\begin{array}{ll}
-\Delta v= f^\sharp & \mbox{in $\Omega^\sharp$},\\\\
\dfrac{\partial v}{\partial \nu}=-\frac{\sqrt{2}}{4}(1+\epsilon) & \mbox{on $\partial\Omega^\sharp$},\\\\
v=\sqrt{\frac{1+\epsilon}{2}} & \mbox{on $\partial\Omega^\sharp$},
\end{array}
\right.
\end{equation}
where $\epsilon>0, f= \chi_{B_1}+\epsilon\chi_{B_2}$. Then $$ \|u\|_{L^6(\Omega)}^6-\|v\|_{L^6(\Omega^\sharp)}^6\approx 2.494-11.134\epsilon +o(\epsilon)$$
\end{example}
\begin{proof}
 The solutions $u$ and $v$ are positive for the maximum principle. A straightforward calculation shows that condition \eqref{cond_2_g} is satisfied. We have
    \begin{equation*}
    u_{| B_1}=\frac{5-r^2}{4},\qquad u_{| B_2}=\frac{\epsilon}{4}(1-r^2),
    \end{equation*}
    $\Omega^\sharp$ is a ball of radius $\sqrt{2}$, $f^\sharp(r)=1$ for $0\leq r< 1$ and $f^\sharp(r)=\epsilon$ for $1<r< \sqrt{2}$ and
    \begin{equation*}
    v=\left\{
    \begin{array}{ll}
c_1-\frac{r^2}{4} & \mbox{for $0\leq r\leq 1$},\\\\
 c_2+c_3\ln{r}-\frac{\epsilon}{4}r^2& \mbox{for $1<r\leq \sqrt{2}$}.
\end{array}\right.
    \end{equation*}
    The solution $v$ is $C^1$ and by imposing the continuity of $v'$, $v$ and the second boundary condition in \eqref{ex3}, we obtain
        \begin{equation*}
    c_3=\frac{1}{2}(\epsilon-1),\quad c_2=\sqrt{\frac{1+\epsilon}{2}}+\frac{\epsilon}{2}+\frac{1}{2}(1-\epsilon)\ln{\sqrt{2}},\quad c_1=\sqrt{\frac{1+\epsilon}{2}}+\frac{1}{4}(1+\epsilon)+\frac{1}{2}(1-\epsilon)\ln{\sqrt{2}}.
    \end{equation*}
    By means of a first order expansion, we get
    %\frac{61741}{28672}\pi%
    \begin{align*}
       \|u\|_{L^6(\Omega)}^6\approx 6.765+o(\epsilon),\quad \|v\|_{L^6(\Omega^\sharp)}^6\approx 4.271+11.134\epsilon +o(\epsilon).
    \end{align*}
    In conclusion
        \begin{align*}
       \|u\|_{L^6(\Omega)}^6-\|v\|_{L^6(\Omega^\sharp)}^6\approx 2.494-11.134\epsilon +o(\epsilon).
    \end{align*}
\end{proof}

%\begin{example}\label{c=0}
%For $a\in\mathbb{R}^+$, let $u_a, v_a$ be the unique solutions to the problems   
%\begin{equation}\nonumber
%\left\{\begin{array}{ll}
%-\Delta u=f & \mbox{in $\Omega_a$},\\\\
%\dfrac{\partial u}{\partial \nu}=0 & \mbox{on $\partial\Omega_a$},\\\\
%\dint_{\partial\Omega_a}u=0,\\
%\end{array}\right.\qquad
%\left\{
%\begin{array}{ll}
%-\Delta v= f^\sharp & \mbox{in $\Omega_a^\sharp$},\\\\
%\dfrac{\partial v}{\partial \nu}=-\frac{1}{2\sqrt{\pi}} & \mbox{on %$\partial\Omega_a^\sharp$},\\\\
%\dint_{\partial\Omega_a^\sharp}v=0,
%\end{array}
%\right.
%\end{equation}
%where $f(x,y)=-\sgn(y)$ and $\Omega_a=\left(-\frac{1}{2a},\frac{1}{2a}\right)\times %\left(-\frac{a}{2},\frac{a}{2}\right)$. Then, it exists $a_0>0$ such that %$\|u_a\|_{L^1(\Omega)}>\|v_a\|_{L^1(\Omega^\sharp)}$ for $a>a_0$.
%\end{example}
%\begin{proof}
%The solution $u_a$ is the function $u_a=\frac{1}{2}y^2\sgn(y)-\frac{1}{2}ay$. %Evaluating its $L^1$ norm, we get
%    $$\|u_a\|_{L^1(\Omega)}=\frac{2}{a}\int_{-\frac{a}{2}}^0 \left(-\frac{1}{2} y^2 %- \frac{1}{2} a y\right) dy=\frac{a^2}{12}.$$
%    We observe that the symmetrized problem do not depend on the parameter $a$, %because $|\Omega_a^|=1$ and $f(x,y)=\sgn(y)$ is equal to $\pm 1$ on sets of measure %$\frac{1}{2}$ for every $a>0$. Therefore, $v_a=v$ and $\|v\|_{L^1(\Omega^\sharp)}$ %is a positive constant.
%\end{proof}
\noindent Finally, a slightly modified version of an example from \cite{MSP} is taken into consideration. This allows us to prove that, when $c<0$, the comparison result \eqref{compL1} no longer holds if $u$ and $v$ fulfill condition \eqref{cond_1} but are not both positive.
\begin{example}\label{ex5}
For $a, \epsilon\in\mathbb{R}^+$, let $u_{a,\epsilon}, v_{a,\epsilon}$ be the unique solutions to the problems   
\begin{equation}\nonumber
\left\{\begin{array}{ll}
-\Delta u=f & \mbox{in $\Omega_{a,\epsilon}$},\\\\
\dfrac{\partial u}{\partial \nu}=-\epsilon & \mbox{on $\partial\Omega_{a,\epsilon}$},\\\\
\dint_{\partial\Omega_{a,\epsilon}}u=0,\\
\end{array}\right.\qquad
\left\{
\begin{array}{ll}
-\Delta v= f^\sharp & \mbox{in $\Omega_{a,\epsilon}^\sharp$},\\\\
\dfrac{\partial v}{\partial \nu}=c^* & \mbox{on $\partial\Omega_{a,\epsilon}^\sharp$},\\\\
\dint_{\partial\Omega_{a,\epsilon}^\sharp}v=0,
\end{array}
\right.
\end{equation}
where $f(x,y)=-(\sgn(y)+\sgn(x))$ and $\Omega_{a,\epsilon}=\left(-2\epsilon-\frac{1}{2a},\frac{1}{2a}\right)\times \left(-2\epsilon-\frac{a}{2},\frac{a}{2}\right)$. Then, it exists $a_0>0$ such that $\|u_{a,\epsilon}\|_{L^1(\Omega_{a,\epsilon})}>\|v_{a,\epsilon}\|_{L^1(\Omega_{a,\epsilon}^\sharp)}$ for $a>a_0$ and $\epsilon<a_0^{-1}$.
%The function $u(x,y)=\frac{1}{2}y^2\sgn(y)-\left(\epsilon+\frac{a}{2}\right)y+\frac{1}{2}x^2\sgn(x)-\left(\epsilon+\frac{1}{2a}\right)x$ with $a, \epsilon>0$, is a solution of the problem
\end{example}
\begin{proof}
The solution $u_{a, \epsilon}$ can be written as $u_{a, \epsilon}=z_{a, \epsilon}+k(a,\epsilon)$, where $z_{a, \epsilon}=\frac{1}{2}y^2\sgn(y)-\left(\epsilon+\frac{a}{2}\right)y+\frac{1}{2}x^2\sgn(x)-\left(\epsilon+\frac{1}{2a}\right)x$. Since 
\begin{equation*}
\int_{\partial\Omega_{a, \epsilon}}z_{a, \epsilon}\,d\mathcal{H}^1\geq 0,
\end{equation*}
then $k(a,\epsilon)\leq 0$. For $x, y\in\Omega$ such that $x, y\geq 0$, we have 
\begin{align*}
\frac{1}{2}y^2\sgn(y)-\left(\epsilon+\frac{a}{2}\right)y&\leq 0,\\ 
\frac{1}{2}x^2\sgn(x)-\left(\epsilon+\frac{1}{2a}\right)x&\leq 0.
\end{align*}
Therefore
\begin{equation*}
\|u_{a, \epsilon}\|_{L^1(\Omega)}\geq \frac{1}{2a}\int_0^{a/2}\left[\left(\epsilon+\frac{a}{2}\right)y-\frac{1}{2}y^2\sgn(y)\right]dy=\frac{a^2}{48}+\frac{a\epsilon}{16}>\frac{a^2}{48}.
\end{equation*}
We consider the solution $w$ to the problem
\begin{equation*}
\left\{\begin{array}{ll}
-\Delta w=2 & \mbox{in $U$},\\\\
w=0 & \mbox{on $\partial U$},\\
\end{array}\right.
\end{equation*}
where U is the ball centered at the origin with radius 4. It exists $a_1>0$ such that $|\Omega_{a, \epsilon}^\sharp|<4$ for $a>a_1$ and $\epsilon<a_1^{-1}$, then $\Omega_{a, \epsilon}^\sharp\subset U$ and applying the maximum principle, we get $w\geq v_{a, \epsilon}$ in $\Omega_{a, \epsilon}^\sharp$, which implies $\|w\|_{L^1(U)}\geq \|v_{a,\epsilon}\|_{L^1(\Omega_{a,\epsilon}^\sharp)}$.
\end{proof}
\noindent As we can see, the conditions \eqref{cond_1} and \eqref{cond_2} become meaningless when $$\int_{\Omega}f\,dx=0.$$ On the other hand, the comparisons proved before are not generally verified if we take into account the solutions whose traces' mean value equals 0. To this aim, we consider the example given in \cite{MSP}.
\begin{example}\label{c=0}
For $a\in\mathbb{R}^+$, let $u_a, v_a$ be the unique solutions to the problems   
\begin{equation}\nonumber
\left\{\begin{array}{ll}
-\Delta u=f & \mbox{in $\Omega_a$},\\\\
\dfrac{\partial u}{\partial \nu}=0 & \mbox{on $\partial\Omega_a$},\\\\
\dint_{\partial\Omega_a}u=0,\\
\end{array}\right.\qquad
\left\{
\begin{array}{ll}
-\Delta v= f^\sharp & \mbox{in $\Omega_a^\sharp$},\\\\
\dfrac{\partial v}{\partial \nu}=-\frac{1}{4\sqrt{\pi}} & \mbox{on $\partial\Omega_a^\sharp$},\\\\
\dint_{\partial\Omega_a^\sharp}v=0,
\end{array}
\right.
\end{equation}
where $f(x,y)=-\sgn(y)$ and $\Omega_a=\left(-\frac{1}{2a},\frac{1}{2a}\right)\times \left(-\frac{a}{2},\frac{a}{2}\right)$. Then, it exists $a_0>0$ such that $\|u_a\|_{L^1(\Omega)}>\|v_a\|_{L^1(\Omega^\sharp)}$ for $a>a_0$.
\end{example}
\begin{proof}
The solution $u_a$ is the function $u_a=\frac{1}{2}y^2\sgn(y)-\frac{1}{2}ay$. Evaluating its $L^1$ norm, we get
    $$\|u_a\|_{L^1(\Omega)}=\frac{2}{a}\int_{-\frac{a}{2}}^0 \left(-\frac{1}{2} y^2 - \frac{1}{2} a y\right) dy=\frac{a^2}{12}.$$
    We observe that the symmetrized problem do not depend on the parameter $a$, because $|\Omega_a^|=1$ and $f(x,y)=\sgn(y)$ is equal to $\pm 1$ on sets of measure $\frac{1}{2}$ for every $a>0$. Therefore, $v_a=v$ and $\|v\|_{L^1(\Omega^\sharp)}$ is a positive constant.
\end{proof}
\section*{Acknowledgements}
All the authors are members of  Gruppo Nazionale per l’Analisi Matematica, la Probabilità e le loro Applicazioni (GNAMPA) of Istituto Nazionale di Alta Matematica (INdAM). The author Cristina Trombetti has been supported by the Project MiUR PRIN-PNRR 2022: "Linear and Nonlinear PDE’S: New directions and Applications", P2022YFAJH. The author Carlo Nitsch was partially supported by the INdAM-GNAMPA project 2023 "Symmetry and asymmetry in PDEs",  cod. CUP E53C22001930001. The author Antonio Celentano was partially supported by the INdAM-GNAMPA project 2023 "Modelli matematici di EDP per fluidi e strutture e proprieta' geometriche delle soluzioni di EDP", cod. CUP E53C22001930001, and by the INdAM-GNAMPA project 2024 "Problemi frazionari: proprietà quantitative ottimali, simmetria, regolarità", cod. CUP E53C23001670001. 

\bibliographystyle{plain}

\bibliography{Bibliografia}

\end{document}